\documentclass[a4paper,10pt,reqno]{amsart}
\usepackage[utf8x]{inputenc}

 \newtheorem{theorem}{Theorem}[section]
 \newtheorem{corollary}[theorem]{Corollary}
 \newtheorem{lemma}[theorem]{Lemma}
 \newtheorem{proposition}[theorem]{Proposition}
 \theoremstyle{definition}
 
 \theoremstyle{remark}

 \numberwithin{equation}{section}

 \usepackage{mathrsfs,esint,amssymb}

\usepackage[colorlinks,citecolor=blue]{hyperref}

\def\R{{\mathbb R}}
\def\S{{\mathscr S}}
\def\e{{\rm e}}
\def\d{{\rm d}}
\def\ri{{\rm i}}

\def\BMO{{\rm BMO}}

\def\Md{{\mathfrak M}}
\def\:{{\colon}}
\def\<{\langle}
\def\>{\rangle}

\def\be#1{\begin{equation}\label{#1}}
\def\ee{\end{equation}}

\title[Generalised Gagliardo--Nirenberg inequalities] {Generalised Gagliardo--Nirenberg inequalities using weak Lebesgue spaces and BMO}

\author[D.\ S.\ McCormick]{David S.\ McCormick}
\thanks{DSMcC is a member of the Warwick ``MASDOC'' doctoral training centre, which is funded by EPSRC grant EP/HO23364/1. JCR is supported by an EPSRC Leadership Fellowship EP/G007470/1.}
\address{D.\ S.\ McCormick \\
Mathematics Institute \\
University of Warwick \\
Coventry, CV4 7AL \\
United Kingdom}
\email{d.s.mccormick@warwick.ac.uk}

\author[J.\ C.\ Robinson]{James C.\ Robinson}
\address{J.\ C.\ Robinson \\
Mathematics Institute \\
University of Warwick \\
Coventry, CV4 7AL \\
United Kingdom}
\email{j.c.robinson@warwick.ac.uk}

\author[J.\ L.\ Rodrigo]{Jose L.\ Rodrigo}
\address{J.\ L.\ Rodrigo \\
Mathematics Institute \\
University of Warwick \\
Coventry, CV4 7AL \\
United Kingdom}
\email{j.l.rodrigo@warwick.ac.uk}

\date{\today}



\makeatletter
\hypersetup{
	pdftitle = {\@title},
	pdfauthor = {\shortauthors},
	pdfsubject = {2010 MSC: \@subjclass},
	pdfkeywords = {\@keywords},
}
\makeatother

\begin{document}

\begin{abstract}
Using elementary arguments based on the Fourier transform we prove that for $1\le q<p<\infty$ and $s\ge0$ with $s>n(1/2-1/p)$, if $f\in L^{q,\infty}(\R^n)\cap\dot H^s(\R^n)$ then $f\in L^p(\R^n)$ and there exists a constant $c_{p,q,s}$ such that
$$
\|f\|_{L^p}\le c_{p,q,s}\|f\|_{L^{q,\infty}}^\theta\|f\|_{\dot H^s}^{1-\theta},
$$
where $1/p=\theta/q+(1-\theta)(1/2-s/n)$. In particular, in $\R^2$  we obtain the generalised Ladyzhenskaya inequality $\|f\|_{L^4}\le c\|f\|_{L^{2,\infty}}^{1/2}\|f\|_{\dot H^1}^{1/2}$. We also show that for $s=n/2$ the norm in $\|f\|_{\dot H^{n/2}}$ can be replaced by the norm in BMO. As well as giving relatively simple proofs of these inequalities, this paper provides a brief primer of some basic concepts in harmonic analysis, including weak spaces, the Fourier transform, the Lebesgue Differentiation Theorem, and Calderon--Zygmund decompositions.\end{abstract}

\maketitle

\section{Introduction}

For $1\le q<p<\infty$ the Gagliardo--Nirenberg interpolation inequality (Nirenberg  \cite{Niren})
\begin{equation}\label{GN}
\|f\|_{L^p}\le c\|f\|_{L^q}^{\theta}\|f\|_{\dot H^s}^{1-\theta},\qquad\frac{1}{p}=\frac{\theta}{q}+(1-\theta)\left(\frac{1}{2}-\frac{s}{n}\right)
\end{equation}
   is an extremely useful tool in the analysis of many partial differential equations. In particular, in the mathematical theory of the two-dimensional Navier--Stokes equations it is frequently encountered in the form of Ladyzhenskaya's inequality (Ladyzhenskaya \cite{Ladyz})
\begin{equation}\label{realLadyz}
\|f\|_{L^4}\le
c\|f\|_{L^2}^{1/2}\|\nabla f\|_{L^2}^{1/2}.
\end{equation}
This paper provides an introduction to some of the basic ideas of harmonic analysis, as a means of generalising the Gagliardo--Nirenberg inequality in two directions.

First, using only simple properties of the weak $L^p$ spaces (Section \ref{sec:weakLp}) and the Fourier transform (Section \ref{sec:FT}), we show that one can replace the $L^q$ norm on the right-hand side of (\ref{GN}) by the norm in the weak $L^q$ space:
\begin{equation}\label{ourineq}
\|f\|_{L^p}\le c\|f\|_{L^{q,\infty}}^\theta\|f\|_{\dot H^s}^{1-\theta}.
\end{equation}
Along the way we also provide a proof of various forms of Young's inequality for convolutions (Section \ref{sec:Young}) and the endpoint Sobolev embedding $\dot H^s(\R^n)\subset L^p(\R^n)$ for $s=n(1/2-1/p)$, $2<p<\infty$ (Section \ref{sec:EPSE}). To our knowledge the direct proof of (\ref{ourineq}) that we provide here in Section \ref{sec:GGN1} is new.

We note that, in particular, (\ref{ourineq}) provides the following generalisation of the 2D Ladyzhenskaya inequality:
\begin{equation}\label{2DgLadyz}
\|f\|_{L^4}\le c\|f\|_{L^{2,\infty}}^{1/2}\|\nabla f\|_{L^2}^{1/2}.
\end{equation}
We outline at the end of Section \ref{sec:Young} how this inequality is relevant for an analysis of the coupled system
\begin{align*}
-\Delta u+\nabla p&=(B\cdot\nabla)B,\qquad\nabla\cdot u=0,\\
\frac{\d B}{\d t}+\eta\Delta B+(u\cdot\nabla)B&=(B\cdot\nabla)u,\qquad\nabla\cdot B=0,
\end{align*}
on a two-dimensional domain (for full details see McCormick et al.\ \cite{MRR}). This system arises from the theory of magnetic relaxation for the generation of stationary Euler flows (see Moffatt \cite{M}), and was our original motivation for pursuing generalisations of (\ref{realLadyz}) and then of (\ref{GN}).

Related to the case $s=n/2$ in (\ref{GN}), Chen \& Zhu \cite{CZ} (see also Azzam \& Bedrossian \cite{AB}; Dong \& Xian \cite{DX}; Kozono \& Wadade \cite{KW}) obtain the inequality
\begin{equation}\label{AzBed}
\|f\|_{L^p}\le c\|f\|_{L^q}^{q/p}\|f\|_\BMO^{1-q/p},
\end{equation}
where BMO is the space of functions with bounded mean oscillation (see Section \ref{sec:BMO}). This inequality (cf.\ Exercise 7.4.1 in Grafakos \cite{Graf2}) is stronger than (\ref{GN}) since $\|f\|_\BMO\le c\|f\|_{\dot H^{n/2}}$ (see Lemma \ref{Hn2nBMO}). In fact one can obtain a stronger inequality still, weakening the $L^q$ norm on the right-hand side as we did in our transition from (\ref{GN}) to (\ref{ourineq}):
\begin{equation}\label{withBMO}
\|f\|_{L^p}\le c\|f\|_{L^{q,\infty}}^{q/p}\|f\|_{\rm BMO}^{1-q/p}.
\end{equation}
 In Section \ref{sec:GGN2} we adapt the proof used in \cite{CZ} for (\ref{AzBed}) to prove (\ref{withBMO}); their argument makes use of the John--Nirenberg inequality for functions in BMO, which is proved via a Calderon--Zygmund type decomposition (Section \ref{sec:BMO}). This decomposition in turn makes use of the Lebesgue Differentiation Theorem (Theorem \ref{LDT}).




 One can prove (\ref{withBMO}), and a slightly stronger inequality involving Lorentz spaces,
 $$
 \|f\|_{L^{p,1}}\le c\|f\|_{L^{q,\infty}}^{q/p}\|f\|_\BMO^{1-q/p},\qquad 1<q<p<\infty,
 $$
 using the theory of interpolation spaces (as in McCormick et al.\ \cite{MRR}); see Corollary \ref{GNLorentz} (and also Kozono et al.\ \cite{KMW}). For the sake of completeness we briefly recall the theory of interpolation spaces in Section \ref{sec:ispaces} and give a proof of this inequality.

 Since it provides one of the main applications of weak $L^p$ spaces, we include a final section that contains a statement of the Marcinkiewicz interpolation theorem and some of its consequences, including a strengthened form of Young's inequality. A very readable account of all the harmonic analysis included here can be found in the two books by Grafakos \cite{Grafakos,Graf2}.

We note that nowhere in this paper do we attempt to find the optimal constants for our inequalities, and throughout we treat functions defined on the whole of $\R^n$. Similar results for functions on bounded domains are more involved, since one requires carefully tailored extension theorems (see Azzam \& Bedrossian \cite{AB}, for example).

\section{Weak $L^p$ spaces and interpolation}\label{sec:weakLp}

We begin with the definition of the weak $L^p$ spaces and quick proofs of some of their properties. For more details see Chapter 1 of Grafakos \cite{Grafakos}.

For a measurable function $f\:\R^n\to\R$ define the \emph{distribution function of $f$} by
$$
d_f(\alpha)=\mu\{x:\ |f(x)|>\alpha\},
$$
where $\mu(A)$ (or later $|A|$) denotes the Lebesgue measure of a set $A$. It follows using Fubini's Theorem that
\begin{equation}\label{df2Lp}
\|f\|_{L^p}^p=\int_{\R^n}|f(x)|^p\,\d x=p\int_{\R^n}\int_0^{|f(x)|}\alpha^{p-1}\,\d\alpha\,\d x=p\int_0^\infty\alpha^{p-1}d_f(\alpha)\,\d\alpha.
\end{equation}
For $1\le p<\infty$ set
\begin{align*}
\|f\|_{L^{p,\infty}}&=\inf\left\{C:\ d_f(\alpha)\le\frac{C^p}{\alpha^p}\right\}\\
&=\sup\{\gamma d_f(\gamma)^{1/p}:\ \gamma>0\}.
\end{align*}
The space $L^{p,\infty}(\R^n)$ consists of all those $f$ such that $\|f\|_{L^{p,\infty}}<\infty$. It follows immediately from the definition that
\begin{equation}\label{Lpw2df}
f\in L^{p,\infty}(\R^n)\qquad\Rightarrow\qquad
d_f(\alpha)\le \|f\|_{L^{p,\infty}}^p\alpha^{-p}
\end{equation}
and that for any $f$ and $g$
\begin{equation}\label{simple}
d_{f+g}(\alpha)\le d_{f}(\alpha/2)+d_{g}(\alpha/2),
\end{equation}
which implies that
\begin{equation}\label{wktri}
\|f+g\|_{L^{p,\infty}}\le 2(\|f\|_{L^{p,\infty}}+\|g\|_{L^{p,\infty}}).
\end{equation}

The following simple lemma (the proof is essentially that of Chebyshev's inequality) is fundamental and shows that any function in $L^p$ is also in $L^{p,\infty}$.

\begin{lemma}\label{Lp2Lpw}
If $f\in L^p(\R^n)$ then $f\in L^{p,\infty}(\R^n)$ and $\|f\|_{L^{p,\infty}}\le\|f\|_{L^p}$.
\end{lemma}

\begin{proof}
This follows since
\[
  d_f(\alpha)=\int_{\{x:\ |f(x)|>\alpha\}}1\,\d x
  \le\int_{\{x:\ |f(x)|>\alpha\}} \frac{|f(x)|^p}{\alpha^p}\,\d x
  \le\|f\|_{L^p}^p\alpha^{-p}.\qedhere
\]
\end{proof}

While $L^p\subset L^{p,\infty}$, clearly $L^{p,\infty}$ is a larger space than $L^p$: for example,
\begin{equation}\label{x2minus}
|x|^{-n/p}\in L^{p,\infty}(\R^n)
\end{equation}
but this function is not an element of $L^p(\R^n)$.

An immediate indication of why these spaces are useful is given in the following simple result, which shows that in the $L^p$ interpolation inequality
$$
\|f\|_{L^r}\le \|f\|_{L^p}^\theta\|f\|_{L^q}^{1-\theta},\qquad\frac{1}{r}=\frac{\theta}{p}+\frac{1-\theta}{q},
$$
one can replace the Lebesgue spaces on the right-hand side by their weak counterparts.

\begin{lemma}\label{Lpinterp}
Take $1\le p<r<q\le\infty$. If $f\in L^{p,\infty}\cap L^{q,\infty}$ then $f\in L^r$ and
$$
\|f\|_{L^r}\le c_{p,r,q}\|f\|_{L^{p,\infty}}^{\theta}\|f\|_{L^{q,\infty}}^{1-\theta},
$$
where
$$
\frac{1}{r}=\frac{\theta}{p}+\frac{1-\theta}{q}.
$$
If $q=\infty$ we interpret $L^{\infty,\infty}$ as $L^\infty$.
\end{lemma}

\begin{proof}
We give the proof when $q<\infty$; the proof when $q=\infty$ is slightly simpler. If $f\in L^{p,\infty}$ then $d_f(\alpha)\le\|f\|_{L^{p,\infty}}^p\alpha^{-p}$, so for any $x$ we have
\begin{align*}
\|f\|_{L^r}^r&=r\int_0^\infty\alpha^{r-1}d_f(\alpha)\,\d\alpha\\
&\le r\int_0^x\alpha^{r-1}\|f\|_{L^{p,\infty}}^p\alpha^{-p}\,\d\alpha+r\int_x^\infty\alpha^{r-1}\|f\|_{L^{q,\infty}}^q\alpha^{-q}\,\d\alpha\\
&\le\frac{r}{r-p}\|f\|_{L^{p,\infty}}^px^{r-p}+\frac{r}{r-q}\|f\|_{L^{q,\infty}}^qx^{q-r}.
\end{align*}
Now choose
$$
x^{p-q}=\frac{\|f\|_{L^{p,\infty}}^p}{\|f\|_{L^{q,\infty}}^q}
$$
to equalise the dependence of the two terms on the right-hand side on the weak norms.\end{proof}

\section{The Fourier transform}\label{sec:FT}

The Schwartz space $\S$ of rapidly decreasing test functions consists of all $\phi\in C^\infty(\R^n)$ such that
$$
\sup_{x\in\R^n}|x^\beta\partial^\alpha\phi|\le M_{\alpha,\beta}\qquad\mbox{for all}\qquad\alpha,\beta\ge0,
$$
where $\alpha,\beta$ are multi-indices.

For any $f\in\S$ one can define the Fourier transform\footnote{There are various possible definitions of the Fourier transform. For example, one could omit the factor of $2\pi$ from the exponential and include a multiplicative factor of $(2\pi)^{-n/2}$ in front of the integral; in this case one keeps the Fourier inversion formula unchanged. However, the fact that the function $\e^{-\pi|x|^2}$ has norm one and is unaffected by the Fourier transform as defined in (\ref{fourierint}) is useful; one can use this to prove the Fourier inversion formula, see Theorem 2.2.14 in Grafakos \cite{Grafakos}, for example.}
\begin{equation}\label{fourierint}
{\mathscr F}[f](\xi)=\hat f(\xi)=\int_{\R^n}\e^{-2\pi\ri\xi\cdot x}f(x)\,\d x.
\end{equation}
It is straightforward to check that
$$
{\mathscr F}[\partial^\alpha f](\xi)=(2\pi\ri)^{|\alpha|}\xi^\alpha\hat f(\xi)\qquad\mbox{and}\qquad{\mathscr F}[x^\beta f](\xi)=(-2\pi\ri)^{|\beta|}[\partial^\beta\hat f](\xi),
$$
from which it follows that $\mathscr F$ maps $\S$ into itself.

Given the Fourier transform of $f$, one can reconstruct $f$ by essentially applying the Fourier transform operator once more:
\be{finversion}
f(x)=\int_{\R^n}\e^{2\pi\ri\xi\cdot x}\hat f(\xi)\,\d\xi.
\ee
If we define $\sigma(f)$ by $\sigma(f)(x)=f(-x)$ then we can write the inversion formula more compactly as $f={\sigma\circ\mathscr F}(\hat f)$. We define ${\mathscr F}^{-1}=\sigma\circ\mathscr F$, the point being that when we can meaningfully extend the definition of $\mathscr F$ and $\sigma$ we will retain this inversion formula.

An obvious extension of the Fourier transform is to any function $f\in L^1(\R^n)$, using the integral definition in (\ref{fourierint}) directly. Since
$$
|\hat f(\xi)|\le\int_{\R^n}|f(x)|\,\d x=\|f\|_{L^1}
$$
it follows that $\mathscr F$ maps $L^1$ into $L^\infty$. Furthermore, there is a natural definition of the Fourier transform for $f\in L^2(\R^n)$. Given $f\in\S$,
\begin{align*}
\|\hat f\|_{L^2}&=\int_{\R^n}\overline{\hat f(x)}\left(\int_{\R^n}\e^{-2\pi\ri\xi\cdot x}f(\xi)\,\d \xi\right)\,\d x\\
&=\int_{\R^n}f(\xi)\left(\int_{\R^n}\overline{\hat f(x)\e^{2\pi\ri\xi\cdot x}}\,\d x\right)\,\d\xi\\
&=\int\overline{f(\xi)}f(\xi)\,\d\xi=\|f\|_{L^2}^2.
\end{align*}
Now given any $f\in L^2$, one can write $f=\lim_{n\to\infty}f_n$, where $f_n\in\S$ and the limit is taken in $L^2$. It follows that $\hat f_n$ is Cauchy in $L^2$, and we identify its limit as $\hat f$. So we can define $\mathscr F\:L^2\to L^2$, with $\|\hat f\|_{L^2}=\|f\|_{L^2}$.

The Fourier transform can therefore be defined (by linearity) for any $f\in L^1+L^2$; $f$ can be recovered from $\hat f$ using ${\mathscr F}^{-1}$ if $\hat f\in L^1+L^2$, and if $\hat f\in L^1$ (in particular if $\hat f\in\S$) then we can use the integral form of the Fourier inversion formula (\ref{finversion}) to give $f$ pointwise as an integral involving $\hat f$.

Given this, we can in fact define the Fourier transform if $f\in L^{r,\infty}$ for some $1<r<2$ (and in particular if $f\in L^r$), by splitting $f$ into two parts, one in $L^1$ and one in $L^2$. The following lemma gives a more general version of this, which will be useful later. We use $\chi_P$ to denote the characteristic function of the set $\{x:\ P\mbox{ holds}\}$.

\begin{lemma}\label{splitg}
  Take $1\le t<r<s\le\infty$, and suppose that $g\in L^{r,\infty}$. For any $M>0$ set
  $$
  g_{M-}=g\chi_{|g|\le M}\qquad\mbox{and}\qquad g_{M+}=g\chi_{|g|>M}.
  $$
  Then $g=g_{M-}+g_{M+}$, where $g_{M-}\in L^s$ with
  \be{g-Ls}
  \|g_{M-}\|_{L^s}^s\le\frac{s}{s-r}M^{s-r}\|g\|_{L^{r,\infty}}^r-M^sd_g(M)
  \ee
  if $s<\infty$ and $\|g_{M-}\|_{L^\infty}\le M$, and $g_{M+}\in L^t$ with
  \be{g+Lt}
  \|g_{M+}\|_{L^t}^t\le\frac{r}{r-t}M^{t-r}\|g\|_{L^{r,\infty}}^r.
  \ee
\end{lemma}

\begin{proof}
Simply note that
\begin{equation}\label{dg-}
d_{g_{M-}}(\alpha)=\begin{cases}
0&\alpha\ge M\\ d_g(\alpha)-d_g(M)&\alpha<M
\end{cases}
\end{equation}
and
\begin{equation}\label{dg+}
d_{g_{M+}}(\alpha)=\begin{cases}
d_g(\alpha)&\alpha>M\\
d_g(M)&\alpha\le M.
\end{cases}
\end{equation}
Then using (\ref{df2Lp}), (\ref{dg-}), and (\ref{Lpw2df}) it is simple to show (\ref{g-Ls}), and (\ref{g+Lt}) follows similarly, using (\ref{dg+}) in place of (\ref{dg-}).
\end{proof}

It is natural to ask what one can say about $\hat f$ when $f\in L^p$. We will see in Section \ref{sec:Marcin} that $\hat f\in L^q$ with $(p,q)$ conjugate, provided that $1\le p\le 2$ (Corollary \ref{cor:FT}).  Note, however, that for any $p>2$ one can find a function in $L^p$ whose Fourier transform is not even a locally integrable function (see Exercise 2.3.13 in Grafakos \cite{Grafakos}).

One can extend the definition further to the space of tempered distributions $\S'$. We say that a sequence $\{\phi_n\}\in\S$ converges to $\phi\in\S$ if
$$
\sup_{x\in\R^n}|x^\alpha\partial^\beta(\phi_n-\phi)|\to0\qquad\mbox{for all}\qquad\alpha,\beta\ge0,
$$
and a linear functional $F$ on $\S$ is an element of $\S'$ if $\<F,\phi_n\>\to \<F,\phi\>$ whenever $\phi_n\to\phi$ in $\S$. It is easy to show that for any $\phi,\psi\in\S$
$$
\<\phi,\hat\psi\>=\<\hat\phi,\psi\>,
$$
and this\footnote{We use $\<\cdot,\cdot\>$ for the action of an element of $\S'$ on elements of $\S$, and set $\<f,g\>=\int fg$ when $f$ and $g$ are functions.} allows us to define the Fourier transform for $F\in\S'$ by setting
$$
\<\hat F,\psi\>=\<F,\hat\psi\>\qquad\mbox{for every}\quad\psi\in\S.
$$
Since one can also extend the definition of $\sigma$ to $\S'$  via the definition $\<\sigma(F),\psi\>=\<F,\sigma(\psi)\>$, the identity $F={\mathscr F}^{-1}\hat F$ still holds in this generality.

\section{Convolution and Young's inequality}\label{sec:Young}

Expressions given by convolutions, i.e.
$$
[f\star g](x)=\int_{\R^n}f(y)g(x-y)\,\d y,
$$
occur frequently. It is a fundamental result that $[f\star g]\hat{\ }(\xi)=\hat f(\xi)\hat g(\xi)$; for $f,g\in\S$ this is the result of simple calculation, which can be extended to $f\in\S$, $g\in\S'$ via the definition $\<f\star g,\phi\>=\<g,\sigma(f)\star\phi\>$.

 One of the primary results for convolutions is Young's inequality.  Following Grafakos (Theorem 1.2.12 in \cite{Grafakos}) we give an elementary proof that
uses only H\"older's inequality.

\begin{lemma}[Young's inequality]\label{sYoung}
Let $1\le p,q,r\le\infty$ satisfy
$$
\frac{1}{p}+1=\frac{1}{q}+\frac{1}{r}.
$$
Then for all $f\in L^q$, $g\in L^r$, we have $f\star g\in L^p$ with
\be{standardYoung}
\|f\star g\|_{L^p}\le\|f\|_{L^q}\|g\|_{L^r}.
\ee
\end{lemma}

\begin{proof} We use $p'$ to denote the conjugate of $p$. Then we have
  $$
  \frac{1}{r'}+\frac{1}{p}+\frac{1}{q'}=1,\qquad\frac{q}{p}+\frac{q}{r'}=1,\quad\mbox{and}\quad \frac{r}{p}+\frac{r}{q'}=1.
  $$
  First use H\"older's inequality with exponents $r'$, $p$, and $q'$:
  \begin{align*}
  |(f\star g)(x)|&\le\int|f(y)||g(x-y)|\,\d y\\
  &=\int|f(y)|^{q/r'}\left(|f(y)|^{q/p}|g(x-y)|^{r/p}\right)|g(x-y)|^{r/q'}\,\d y\\
  &\le\|f\|_{L^q}^{q/r'}\left(\int|f(y)|^q|g(x-y)|^r\,\d y\right)^{1/p}\left(\int|g(x-y)|^r\,\d y\right)^{1/q'}\\
  &\le\|f\|_{L^q}^{q/r'}\left(\int|f(y)|^q|g(x-y)|^r\,\d y\right)^{1/p}\|g\|_{L^r}^{r/q'}.
  \end{align*}
Now take the $L^p$ norm (with respect to $x$):
\begin{align*}
\|f\star g\|_{L^p}&\le\|f\|_{L^q}^{q/r'}\|g\|_{L^r}^{r/q'}\left(\iint|f(y)|^q|g(x-y)|^r\,\d y\,\d x\right)^{1/p}\nonumber\\
&=\|f\|_{L^q}^{q/r'}\|g\|_{L^r}^{r/q'}\|f\|_{L^q}^{q/p}\|g\|_{L^r}^{r/p}\nonumber\\
&=\|f\|_{L^q}\|g\|_{L^r}.\qedhere
\end{align*}
\end{proof}

We will need a version of this inequality that allows $L^q$ on the right-hand side to be replaced by $L^{q,\infty}$. The price we have to pay for this (at least initially) is that we also weaken the left-hand side; and note that we have also lost the possibility of some endpoint values ($r=\infty$ and $p,q=1,\infty$) that are allowed in (\ref{standardYoung}). In fact one can keep the full $L^p$ norm on the left, provided that $r>1$; but this requires Proposition \ref{prop:wYoung} as an intermediate step and the Marcinkiewicz Interpolation Theorem (see Section \ref{sec:Marcin}).

\begin{proposition}\label{prop:wYoung}
Suppose that $1\le r<\infty$ and $1<p,q<\infty$. If $f\in L^{q,\infty}$ and $g\in L^r$ with
$$
\frac{1}{p}+1=\frac{1}{q}+\frac{1}{r}
$$
then $f\star g\in L^{p,\infty}$ with
\begin{equation}\label{w2wYoung}
\|f\star g\|_{L^{p,\infty}}\le c_{p,q,r}\|f\|_{L^{q,\infty}}\|g\|_{L^r}.
\end{equation}
\end{proposition}

\begin{proof}
We follow the proof in Grafakos \cite{Grafakos}, skipping some of the algebra. We have already introduced the main step, the splitting of $f$ in Lemma \ref{splitg}. For a fixed $M>0$ we set $f=f_{M-}+f_{M+}$. Using (\ref{g-Ls}) and H\"older's inequality we obtain
$$
|(f_{M-}\star g)(x)|\le\|f_{M-}\|_{L^{q'}}\|g\|_{L^q}\le \left(\frac{q'}{q'-r}M^{q'-r}\|f\|_{L^{r,\infty}}^r\right)^{1/q'}\|g\|_{L^q},
$$
where $(q,q')$ are conjugate; the right-hand side reduces to $M\|g\|_{L^1}$ if $q=1$. Note in particular that if
$$
M=(\alpha^{q'}2^{-q'}rp^{-1}\|f\|_{L^{r,\infty}}^{-r}\|g\|_{L^q}^{-q'})^{1/(q'-r)}
$$
(or $\alpha/2\|g\|_{L^1}$ if $q=1$) then $d_{f_{M-}\star g}(\alpha/2)=0$.

For $f_{M+}$ we can use (\ref{g+Lt}) and apply Young's inequality to yield
$$
\|f_{M+}\star g\|_{L^q}\le\|f_{M+}\|_{L^1}\|g\|_{L^q}\le\frac{r}{r-1}M^{1-r}\|f\|_{L^{r,\infty}}^r\|g\|_{L^q}.
$$

Choosing $M$ as above and using (\ref{simple}) it follows that
\begin{align*}
d_{f\star g}(\alpha)&\le d_{f_{M+}\star g}(\alpha/2)\\
&\le(2\|f_{M+}\star g\|_{L^p}\alpha^{-1})^q\\
&\le(2rM^{1-r}\|f\|_{L^{r,\infty}}^r\|g\|_{L^q}(r-1)^{-1}\alpha^{-1})^q\\
&=C\|f\|_{L^{r,\infty}}^p\|g\|_{L^q}^p\alpha^{-p},
\end{align*}
which yields (\ref{w2wYoung}).\end{proof}

This result has implications, among other things, for the regularity of solutions of elliptic equations. It was mentioned in the introduction that our study of generalised Gagliardo--Nirenberg inequalities was motivated by the study of a particular coupled system in two dimensions, namely
\begin{align*}
-\Delta u+\nabla p&=(B\cdot\nabla)B,\qquad\nabla\cdot u=0,\\
\frac{\d B}{\d t}+\eta\Delta B+(u\cdot\nabla)B&=(B\cdot\nabla)u,\qquad\nabla\cdot B=0.
\end{align*}
Formal energy estimates (which can be made rigorous via a suitable regularisation) yield
$$
\frac{1}{2}\|B(t)\|_{L^2}^2+\eta\int_0^t\|\nabla B\|_{L^2}^2+\int_0^t\|\nabla u\|_{L^2}^2\le \frac{1}{2}\|B(0)\|_{L^2}^2,
$$
showing in particular that $B\in L^\infty(0,T;L^2)$ when $B(0)\in L^2$. To obtain a similar uniform estimate on $u$ we need to understand the regularity of solutions of the Stokes problem
$$
-\Delta u+\nabla p=(B\cdot\nabla)B\qquad\nabla\cdot u=0
$$
when $B\in L^2$. A slightly simpler problem with the same features is
\be{RHSdiv}
-\Delta \phi=\partial_if,
\ee
with $f\in L^1$. It is well known that the solution of $-\Delta\phi=g$ in $\R^2$ is given by $E\star g$, where
$$
E(x)=-\frac{1}{2\pi}\log|x|.
$$
Noting (after an integration by parts) that the solution of (\ref{RHSdiv}) is given by $\partial_i E\star f$, and that $\partial_i E\in L^{2,\infty}$, it follows from Proposition \ref{prop:wYoung} that $f\in L^1$ implies that $\phi\in L^{2,\infty}$. [The stronger version of Young's inequality given in Theorem \ref{bestYoung} does not apply when $f\in L^1$, so would not improve the regularity here.] Thus to obtain further estimates (in particular on the time derivative of $B$) we required a version of the Ladyzhenskaya inequality that replaced the $L^2$ norm of $u$ with the norm of $u$ in $L^{2,\infty}$. Further details can be found in McCormick et al.\ \cite{MRR}.

\section{Endpoint Sobolev embedding}\label{sec:EPSE}

In our proof of the inequality
$$
\|f\|_{L^p}\le c\|f\|_{L^{q,\infty}}^\alpha\|f\|_{\dot H^s}^{1-\alpha}
$$
we will use the endpoint Sobolev embedding $\dot H^s(\R^n)\subset L^p(\R^n)$ for $s=n(1/2-1/p)$ when $2<p<\infty$. We prove this here, following Theorem 1.2 in Chemin et al.\ \cite{CDG*}.

Since the Fourier transform maps $L^2$ isometrically into itself, and
$$
{\mathscr F}[\partial^\alpha f](\xi)=(2\pi\ri)^{|\alpha|}\xi^\alpha\hat f(\xi),
$$
it is relatively straightforward to show that when $s$ is a non-negative integer
\begin{equation}\label{SobnormFT}
\sum_{|\alpha|=s}\|\partial^\alpha f\|_{L^2}^2\ \simeq\ \int_{\R^n}|\xi|^{2s}|\hat f(\xi)|^2\,\d\xi,
\end{equation}
where we write $a\simeq b$ if there are constants $0<c\le C$ such that $ca\le b\le Ca$.

For any $s\ge0$, even if $s$ is not an integer, we can define\footnote{We follow the definition of Bahouri et al.\ \cite{C**} (see also Chemin et al.\ \cite{CDG*}), including the condition that $\hat f\in L^1_{\rm loc}(\R^n)$. This sidesteps complexities that arise from problems with understanding the meaning of $|\xi|^s\hat f$ if one only knows that $\hat f\in\S'$; see the discussion in Chapter 6 of Grafakos \cite{Graf2}.} the homogeneous Sobolev space $\dot H^s(\R^n)$ using (\ref{SobnormFT}):
$$
\dot H^s(\R^n)=\left\{f\in\S':\ \hat f\in L^1_{\rm loc}(\R^n)\mbox{ and }\int_{\R^n}|\xi|^{2s}|\hat f(\xi)|^2\,\d\xi<\infty\right\}.
$$
For $s<n/2$ this is a Hilbert space with the natural norm
$$
\|f\|_{\dot H^s}=\left(\int_{\R^n}|\xi|^{2s}|\hat f(\xi)|^2\,\d\xi\right)^{1/2},
$$
and one can therefore also define $\dot H^s(\R^n)$ in this case as the completion of $\S$ with respect to the $\dot H^s$ norm (that $\dot H^s(\R^n)$ is complete iff $s<n/2$ is shown in Bahouri et al.\ \cite{C**}; the simple example showing that $\dot H^s(\R^n)$ is not complete when $s\ge n/2$ can also be found in Chemin et al.\ \cite{CDG*}).

\begin{theorem}\label{EPS}
For $2<p<\infty$ there exists a constant $c=c_{n,p}$ such that if $f\in\dot H^s(\R^n)$ with $s=n(1/2-1/p)$ then $f\in L^p(\R^n)$ and
\begin{equation}\label{endpt}
\|f\|_{L^p}\le c\|f\|_{\dot H^s}.
\end{equation}
\end{theorem}

\begin{proof} First we prove the result when $\|f\|_{\dot H^s}=1$. For such an $f$, write $f=f_{<R}+f_{>R}$, where
\begin{equation}\label{flgR}
f_{<R}={\mathscr F}^{-1}(\hat f\chi_{\{|\xi|\le R\}})\qquad\mbox{and}\qquad f_{>R}={\mathscr F}^{-1}(\hat f\chi_{\{|\xi|>R\}}).
\end{equation}
In both expressions the Fourier inversion formula makes sense: for $f_{>R}$ we know that $\hat f\chi_{>R}\in L^2(\R^n)$, and $\mathscr F$ (and likewise ${\mathscr F}^{-1}$) is defined on $L^2$; while for $f_{<R}$ we know that $\hat f\in L^1_{\rm loc}(\R^n)$, and so $\hat f\chi_{\le R}\in L^1(\R^n)$ which means that we can write $f_{<R}$ using the integral form of the inversion formula
(\ref{finversion}) to write
$$
f_{<R}(x)=\int_{|\xi|\le R} \e^{2\pi\ri \xi\cdot x}\hat f(\xi)\,\d\xi.
$$

Thus
\begin{align*}
\|f_{<R}\|_{L^\infty}&\le\int_{|\xi|\le R}|\xi|^{-s}|\xi|^s|\hat f(\xi)|\,\d\xi\\
&\le\left(\int_{|\xi|\le R}|\xi|^{-2s}\,\d\xi\right)^{1/2}\|f\|_{\dot H^s}= C_sR^{n/2-s}=C_sR^{n/p},
\end{align*}
since we took $\|f\|_{\dot H^s}=1$ and $s=n(\frac{1}{2}-\frac{1}{p})$. Now, since for any choice of $R$
$$
d_f(\alpha)\le d_{f_{<R}}(\alpha/2)+ d_{f_{>R}}(\alpha/2)
$$
(using (\ref{simple})), we can choose $R$ to depend on $\alpha$,  $R=R_\alpha:=(\alpha/2C_s)^{p/n}$, and then we have
$$
d_{f_{<R_\alpha}}(\alpha/2)=0,
$$
it follows that $d_f(\alpha)\le d_{f_{>R_\alpha}}(\alpha/2)$. Thus, using the fact that the Fourier transform is an isometry from $L^2$ into itself,
\begin{align*}
\|f\|_{L^p}^p&\le p\int_0^\infty\alpha^{p-1}d_{f_{>R_\alpha}}(\alpha/2)\,\d\alpha\\
&\le p\int_0^\infty\alpha^{p-1}\frac{4}{\alpha^2}\|f_{>R_\alpha}\|_{L^2}^2\,\d\alpha\\
&= C\int_0^\infty \alpha^{p-3}\|{\mathscr F}(f_{>R_\alpha})\|_{L^2}^2\,\d\alpha\\
&= C\int_0^\infty \alpha^{p-3}\int_{|\xi|\ge R_\alpha}|\hat f(\xi)|^2\,\d\xi\,\d\alpha\\
& =C\int_{\R^n}\left(\int_0^{2C_s|\xi|^{n/p}}\alpha^{p-3}\,\d\alpha\right)|\hat f(\xi)|^2\,\d\xi\\
&\le C\int_{\R^n}|\xi|^{n(p-2)/p}|\hat f(\xi)|^2\,\d s\\
& = C,
\end{align*}
since $n(p-2)/p=2s$ and we took $\|f\|_{\dot H^s}=1$.

Thus for $f\in\dot H^s$ with $\|f\|_{\dot H^s}=1$ we have $\|f\|_{L^p}\le C$, and (\ref{endpt}) follows for general $f\in\dot H^s$ on applying this result to $g=f/\|f\|_{\dot H^s}$.\end{proof}

\section{A weak-strong Bernstein inequality}\label{sec:bernstein}

In the next section we will require a result, known as Bernstein's inequality, that provides integrability of $f$ assuming localisation of its Fourier transform: if $\hat f$ is supported in $B(0,R)$ (the ball of radius $R$) then for any $1\le p\le q\le \infty$ if $f\in L^p(\R^n)$ then
\begin{equation}\label{normalB}
\|f\|_{L^q}\le c_{p,q}R^{n(1/p-1/q)}\|f\|_{L^p}.
\end{equation}
For our purposes we will require a version of this inequality that replaces $L^p$ by $L^{p,\infty}$ on the right-hand side.

As in the standard proof of (\ref{normalB}), we make use of the following simple result. We use the notation ${\mathfrak D}_hf(x)=h^{-n}f(x/h)$; note that $\widehat{{\mathfrak D}_h}(x)=\hat f(hx)$. The support of $g\in\S'$ is the intersection of all closed sets $K$ such that $\<g,\phi\>=0$ whenever the support of $\phi\in\S$ is disjoint from $K$.

\begin{lemma}
  There is a fixed $\phi\in\S$ such that if $\hat f$ is supported in $B(0,R)$ then $f=({\mathfrak D}_{1/R}\phi)\star f$.
\end{lemma}

\begin{proof}
  Take $\phi\in{\mathscr S}$ so that $\hat\phi=1$ on $B(0,1)$. Then
  $$
  \widehat{{\mathfrak D}_{1/R}\phi}(\xi)=\hat\phi(\xi/R)
  $$
  which is equal to $1$ on $B(0,R)$. Thus $({\mathfrak D}_{1/R}\phi)\star f-f$ has Fourier transform zero, and the lemma follows.
\end{proof}

For use in the proof of our next lemma, note that
\begin{equation}\label{DphiLp}
\|{\mathfrak D}_{1/R}\phi\|_{L^r}=R^{n(1-1/r)}\|\phi\|_{L^r}.
\end{equation}

\begin{lemma}[Weak-strong Bernstein inequality]\label{Bernstein}
Let $1\le p<\infty$ and suppose that $f\in L^{p,\infty}(\R^n)$ and that $\hat f$ is supported in $B(0,R)$. Then for each $q$ with $p<q<\infty$ there exists a constant $c_{p,q}$ such that
\begin{equation}\label{bernstein}
\|f\|_{L^{q}}\le cR^{n(1/p-1/q)}\|f\|_{L^{p,\infty}}.
\end{equation}
\end{lemma}

\begin{proof}
We follow the standard proof, replacing Young's inequality by its weak form, and making use of the interpolation result of Lemma \ref{Lpinterp}. First we prove the weak-weak version
$$
\|f\|_{L^{q,\infty}}\le cR^{n(1/p-1/q)}\|f\|_{L^{p,\infty}}
$$
valid for all $1\le p\le q<\infty$. To do this we simply apply the weak form of Young's inequality (Proposition \ref{prop:wYoung}) to $f=\phi^{1/R}\star f$:
\begin{align*}
\|f\|_{L^{q,\infty}}&=\|({\mathfrak D}_{1/R}\phi)\star f\|_{L^{q,\infty}}\\
&\le c\|{\mathfrak D}_{1/R}\phi\|_{L^r}\|f\|_{L^{p,\infty}},
\end{align*}
where
$$
1+\frac{1}{q}=\frac{1}{r}+\frac{1}{p}
$$
with $1\le p<\infty$ and $1<q,r<\infty$. It follows using (\ref{DphiLp}) that
  $$
  \|f\|_{L^{1,\infty}}\le cR^{n(1/p-1)}\|f\|_{L^{p,\infty}}
 \qquad\mbox{and}\qquad
  \|f\|_{L^{2q,\infty}}\le cR^{n(1/p-1/2q)}\|f\|_{L^{p,\infty}},
  $$
 and we then obtain (\ref{bernstein}) by interpolation of $L^q$ between $L^{1,\infty}$ and $L^{2q,\infty}$ (Lemma \ref{Lpinterp}),
    \begin{align}
    \|f\|_{L^q}&\le c\|f\|_{L^{1,\infty}}^{1/(2q-1)}\|f\|_{L^{2q,\infty}}^{(2q-2)/(2q-1)}\nonumber\\
    &\le cR^{n(1/p-1/q)}\|f\|_{L^{p,\infty}}.\qedhere
    \end{align}
\end{proof}

\section{Generalised Gagliardo--Nirenberg inequality I}\label{sec:GGN1}


We now prove our first generalisation of the Gagliardo--Nirenberg inequality, replacing the $L^q$ norm on the right-hand side of (\ref{GN}) by the norm in $L^{q,\infty}$. The new part of the following result is when $s\ge n/2$, with the case $s=n/2$ particularly interesting: in the range $n(1/2-1/p)<s<n/2$ the inequality follows using weak-$L^p$ interpolation from Lemma \ref{Lpinterp} coupled with the Sobolev embedding $\dot H^{n(1/2-1/p)}\subset L^p$ from Theorem \ref{EPS}.

\begin{theorem}\label{thm:Ladyz}
Take $1\le q<p$ and $s\ge0$ with $s>n(1/2-1/p)$. There exists a constant $c_{p,q,s}$  such that if $f\in L^{q,\infty}(\R^n)\cap \dot H^s(\R^n)$ then $f\in L^p(\R^n)$ and
  \begin{equation}\label{interpolated}
  \|f\|_{L^p}\le c_{p,q,s}\|f\|_{L^{q,\infty}}^{\theta}\|f\|_{\dot H^s}^{1-\theta}\qquad\mbox{for every}\quad f\in L^{q,\infty}\cap\dot H^s,
  \end{equation}
  where
  \begin{equation}\label{alphais}
  \frac{1}{p}=\frac{\theta}{q}+(1-\theta)\left(\frac{1}{2}-\frac{s}{n}\right).
  \end{equation}
\end{theorem}

\begin{proof}
%
First we prove the theorem in the case $p\ge 2$. As in the proof of Theorem \ref{EPS} we write
$$
f=f_{<R}+f_{>R},
$$
where $f_{<R}$ and $f_{>R}$ are defined in (\ref{flgR}).

Using the endpoint Sobolev embedding $\dot H^{n(1/2-1/p)}(\R^n)\subset L^p(\R^n)$ from Theorem \ref{EPS} (taking $\dot H^0=L^2$ when $p=2$) we can estimate
\begin{align*}
\|f_{>R}\|_{L^p}&\le c\|f_{>R}\|_{\dot H^{n(1/2-1/p)}}\\
&=c\left(\int_{|\xi|\ge R}|\xi|^{2n(1/2-1/p)}|\hat f(\xi)|^2\,\d\xi\right)^{1/2}\\
&\le\frac{c}{R^{s-n(1/2-1/p)}}\left(\int_{|\xi|\ge R}|\xi|^{2s}|\hat f(\xi)|^2\,\d\xi\right)^{1/2}\\
&=\frac{c}{R^{s-n(1/2-1/p)}}\|f\|_{\dot H^s},
\end{align*}
while
$$
\|f_{<R}\|_{L^p}\le cR^{n(1/q-1/p)}\|f_{<R}\|_{L^{q,\infty}}\le cR^{n(1/q-1/p)}\|f\|_{L^{q,\infty}}
$$
using the weak-strong Bernstein inequality from Lemma \ref{Bernstein} and (\ref{wktri}).

Thus
$$
\|f\|_{L^p}\le c(R^{n(1/q-1/p)}\|f\|_{L^{q,\infty}}+R^{-s+n(1/2-1/p)}\|f\|_{\dot H^s}).
$$
Choosing
$$
R^{s+n(1/q-1/2)}=\frac{\|f\|_{\dot H^s}}{\|f\|_{L^{q,\infty}}}
$$
we obtain
\begin{equation}\label{halfI}
\|f\|_{L^p}\le c\|f\|_{L^{q,\infty}}^{\theta}\|f\|_{\dot H^s}^{1-\theta},
\end{equation}
where
$$
\theta=1-n\frac{1/q-1/p}{s+n(1/q-1/2)},
$$
which on rearrangement yields the condition (\ref{alphais}).

If $1\le q<p<2$ then we first interpolate $L^p$ between $L^{q,\infty}$ and $L^2$, and then use the above result with $p=2$. Setting $\frac{1}{2}=\frac{\theta'}{q}+(1-\theta')\left(\frac{1}{2}-\frac{s}{n}\right)$ we have
\begin{align*}
\|f\|_{L^p}&\le c\|f\|_{L^{q,\infty}}^{q(2-p)/p(2-q)}\|f\|_{L^2}^{2(p-q)/p(2-q)}\\
&\le c\|f\|_{L^{q,\infty}}^{q(2-p)/p(2-q)}\left(c\|f\|_{L^{q,\infty}}^{\theta'}\|f\|_{\dot H^s}^{1-\theta'}\right)^{2(p-q)/p(2-q)}\\
&=c\|f\|_{L^{q,\infty}}^\theta\|f\|_{\dot H^s}^{1-\theta},
\end{align*}
with $\theta$ given by (\ref{alphais}), as required.\end{proof}

\section{The space BMO of functions with bounded mean oscillation}\label{sec:BMO}

For any set $A\subset\R^n$ we write
$$
f_A=\frac{1}{|A|}\int_{A} f\,\d x
$$
for the average of $f$ over the set $A$. The space of functions with bounded mean oscillation, $\BMO(\R^n)$, consists of those functions $f$ for which
$$
\|f\|_\BMO:=\sup_{Q\subset\R^n}\frac{1}{|Q|}\int_Q|f-f_Q|\,\d x
$$
is finite, where the supremum is taken over all cubes $Q\subset\R^n$. Note that this is a not a norm (any constant function has $\|c\|_\BMO=0$), but BMO is a linear space, i.e.\ if $f,g\in\BMO$ then  $f+g\in\BMO$ and
$$
\|f+g\|_\BMO\le\|f\|_\BMO+\|g\|_\BMO.
$$
This space was introduced by John \& Nirenberg \cite{JN}; more details can be found in Chapter 7 of Grafakos \cite{Graf2}, for example.

BMO is a space with the same scaling as $L^\infty$, but is a larger space. Indeed, if $f\in L^\infty(\R^n)$ then clearly for any cube $Q$
\be{fmfQf}
\int_Q|f-f_Q|\,\d x\le2\int_Q|f|\le 2|Q|\|f\|_{L^\infty},
\ee
and so
\be{Linf2BMO}
\|f\|_\BMO\le2\|f\|_{L^\infty}.
\ee
However, the function $\log|x|\in\BMO(\R^n)$ but is not bounded on $\R^n$ (Example 7.1.3 in Grafakos \cite{Graf2}).

The endpoint Sobolev embedding from Theorem \ref{EPS} fails when $s=n/2$, but at this endpoint we still have $\dot H^{n/2}(\R^n)\subset\BMO(\R^n)$. This is simple to show (following Theorem 1.48 in Bahouri et al.\ \cite{C**}), if we note that
for any $x\in Q$
$$
|f(x)-f_Q|=\left|\frac{1}{|Q|}\int_Q f(x)-f(y)\,\d y\right|\le\sqrt{n}|Q|^{1/n}\|\nabla f\|_{L^\infty(Q)}.
$$

\begin{lemma}\label{Hn2nBMO}
If $f\in L^1_{\rm loc}(\R^n)\cap\dot H^{n/2}(\R^n)$ then $f\in\BMO(\R^n)$ and there exists a constant $C=C(n)$ such that
$$
\|f\|_\BMO\le C\|f\|_{\dot H^{n/2}}\qquad\mbox{for all}\quad f\in L^1_{\rm loc}(\R^n)\cap\dot H^{n/2}(\R^n).
$$
\end{lemma}

\begin{proof}
We write $f=f_{<R}+f_{>R}$ as in the proof of Theorem \ref{EPS} and then, recalling (\ref{fmfQf}),
\begin{align*}
\frac{1}{|Q|}&\int_Q|f-f_Q|\le\sqrt{n}|Q|^{1/n}\|\nabla f_{<R}\|_{L^\infty(Q)}+\frac{1}{|Q|}\int_Q|f_{>R}-(f_{>R})_Q|\\
&\le\sqrt{n}|Q|^{1/n}\int_{|\xi|\le R}|\xi||\hat f(\xi)|\,\d\xi+\frac{2}{|Q|^{1/2}}\left(\int_Q|f_{<R}|^2\right)^{1/2}\\
&\le\sqrt{n}|Q|^{1/n}cR\left(\int_{\R^n}|\xi|^{n/2}|\hat f(\xi)|^2\,\d\xi\right)^{1/2}+\frac{2}{|Q|^{1/2}}\left(\int_{|\xi|\ge R}|\hat f(\xi)|^2\right)^{1/2}\\
&\le c_n[|Q|^{1/n}R+|Q|^{-1/2}R^{n/2}]\|f\|_{\dot H^{n/2}}.
\end{align*}
Choosing $R=|Q|^{-1/n}$ yields
$$
\frac{1}{|Q|}\int_Q|f-f_Q|\le C\|f\|_{\dot H^{n/2}};
$$
taking the supremum over all cubes $Q\subset\R^n$ yields $\|f\|_\BMO\le C\|f\|_{\dot H^{n/2}}$.\end{proof}


We now want to prove a result, due to John \& Nirenberg \cite{JN}, that gives an important property of functions in BMO that will be crucial in the proof of the inequality
$$
   \|f\|_{L^p}\le C\|f\|_{L^{q,\infty}}^{q/p}\|f\|_\BMO^{1-q/p},\qquad q<p<\infty,
$$
given in the next section. To prove the John--Nirenberg inequality we will need a Calderon--Zygmund type decomposition of $\R^n$ into a family of cubes with certain useful properties. The proof that such a decomposition is possible uses the Lebesgue Differentiation Theorem, which we now state (without proof).

We define the uncentred cubic maximal function by
$$
\Md f(x)=\sup_{Q\ni x}\frac{1}{|Q|}\int_Q|f(y)|\,\d y,
$$
where the supremum is taken over all cubes $Q\subset\R^n$ that contain $x$. The proof of the Lebesgue Differentiation Theorem uses the fact that $\Md$ maps $L^1$ into $L^{1,\infty}$; see Section 3.4 in Folland \cite{F} or Section 2.1 in Grafakos \cite{Grafakos} for details.

\begin{theorem}[Lebesgue Differentiation Theorem]\label{LDT}
If $f\in L^1_{\rm loc}(\R^n)$ then
\begin{equation}\label{LebDiff}
\lim_{|Q|\to0}\frac{1}{|Q|}\int_Qf(y)\,\d y=f(x)
\end{equation}
for almost every $x\in\R^n$, where $Q$ is a cube containing $x$. As a consequence, $|f(x)|\le \Md f(x)$ almost everywhere.
\end{theorem}

\begin{proposition}\label{prop:cubes}
  Let $Q$ be any cube in $\R^n$. Given $f\in L^1(Q)$ and
  $$
  M\ge\frac{1}{|Q|}\int_Q|f|
  $$
  there exists a countable collection $\{Q_j\}$ of disjoint open cubes such that $|f(x)|\le M$ for almost every $x\in Q\setminus\bigcup_jQ_j$ and
  \be{upanddown}
  M<\frac{1}{|Q_j|}\int_{Q_j}|f(x)|\,\d x\le 2^nM
  \ee
  for every $Q_j$.
\end{proposition}

Note that it follows from (\ref{upanddown}) that
\be{sum2int}
\sum_j|Q_j|\le\frac{1}{M}\int_Q|f|.
\ee

\begin{proof}
  Decompose $Q$ by halving each side into a collection ${\mathscr Q}_0$ of $2^n$ equal cubes.
Select one of these cubes $\hat Q$ if
  \begin{equation}\label{badun}
  \frac{1}{|\hat Q|}\int_{\hat Q}|f(x)|\,\d x>M.
  \end{equation}
  Call the selected cubes ${\mathscr C}_1$ and let ${\mathscr Q}_1={\mathscr Q}_0\setminus{\mathscr C}_1$.

  Repeat this process inductively, to produce a set ${\mathscr C}=\bigcup_j{\mathscr C}_j$ of selected cubes, on which (\ref{badun}) holds. Note that if $\hat Q$ was selected at step $k$ then it is contained in a cube $Q'\in{\mathscr Q}_{k-1}$, and so
  $$
  M<\frac{1}{|\hat Q|}\int_{\hat Q}|f(x)|\,\d x\le 2^n\frac{1}{|Q'|}\int_{Q'}|f(x)|\,\d x\le 2^nM.
  $$

  Finally, if $x\in Q\setminus\bigcup_jQ_j$ then there exists a sequences of cubes $Q_k$ containing $x$ with sides shrinking to zero and such that
  $$
  \frac{1}{|Q_k|}\int_{Q_k}|f(x)|\,\d x\le M.
  $$
  It follows from the Lebesgue Differentiation Theorem that $|f(x)|\le M$ for almost every $x\in Q\setminus\bigcup_jQ_j$.
\end{proof}

\begin{lemma}[John--Nirenberg inequality]\label{JN}
There exist constants $c$ and $C$ (depending only on $n$) such that if $f\in\BMO(\R^n)$ then for any cube $Q\subset\R^n$
\be{expJN}
|\{x\in Q:\ |f-f_Q|>\alpha\}|\le \frac{C}{\|f\|_\BMO}\e^{-c\alpha/\|f\|_\BMO}\int_Q|f-f_Q|
\ee
for all $\alpha\ge\|f\|_\BMO$.
\end{lemma}

\begin{proof}
We prove the result assuming that $\|f\|_\BMO=1$; we then obtain (\ref{expJN}) by applying the resulting inequality to $f/\|f\|_\BMO$. Let $F(\alpha)$ be the minimum number such that the inequality
\be{ineq4JN}
|\{x\in Q:\ |f(x)|>\alpha\}|\le F(\alpha)\int_Q|f|
\ee
holds for all $f\in L^1(Q)$ and all cubes $Q$; note (cf.\ Lemma \ref{Lp2Lpw}) that $F(\alpha)\le 1/\alpha$.

Following the original proof of John \& Nirenberg \cite{JN} we show that for all $\alpha\ge 2^n$,
\be{2iterate}
F(\alpha)\le\frac{1}{M}F(\alpha-2^nM)\qquad\mbox{for all}\quad 1\le M\le 2^{-n}\alpha.
\ee
Given $M$ in this range we decompose $f$ using Proposition \ref{prop:cubes}. Now, if $|f(x)|>\alpha\ge 2^n$ then $x\in Q_k$ for some $k$, and we know that $|f_{Q_k}|\le 2^nM$ from (\ref{upanddown}). So then
$$
|\{x\in Q:\ |f(x)|>\alpha\}|\le\sum_k|\{x\in Q_k:\ |f(x)-f_{Q_k}|>\alpha-2^nM\}|.
$$
We can now use (\ref{ineq4JN}) on the cube $Q_k$ for the function $f-f_{Q_k}$, so that
\begin{align*}
|\{x\in Q_k:\ |f(x)-f_{Q_k}|>\alpha-2^nM\}|&\le F(\alpha-2^nM)\int_{Q_k}|f-f_{Q_k}|\,\d x\\
&\le F(\alpha-2^nM)|Q_k|
\end{align*}
(recall that we took $\|f\|_\BMO=1$). It follows using (\ref{sum2int}) that
$$
|\{x\in Q:\ |f(x)|>\alpha\}\le\left(\sum_k|Q_k|\right)F(\alpha-2^nM)\le \frac{1}{M} F(\alpha-2^nM)\int_Q|f|\,\d x,
$$
which is (\ref{2iterate}).

To finish the proof we iterate (\ref{2iterate}) in a suitable way. We remarked above that $F(\alpha)\le1/\alpha$; now observe that
$$
\frac{1}{\alpha}\le C\e^{-\alpha/2^n\e}\qquad 1\le\alpha\le1+2^n\e,
$$
for $C=\max_{1\le\alpha\le1+2^n\e}\alpha^{-1}\e^{\alpha/2^n\e}$. Iterating (\ref{2iterate}) with $M=\e$, which implies that $F(\alpha+2^n\e)\le\frac{1}{\e}F(\alpha)$, we obtain
$$
F(\alpha)\le C\e^{-c\alpha}\qquad\mbox{for all}\quad \alpha\ge 1,
$$
where $c=1/2^n\e$, which gives (\ref{expJN}).
\end{proof}

The more usually quoted form of this inequality,
$$
|\{x\in Q:\ |f-f_Q|>\alpha\}|\le C|Q|\e^{-c\alpha/\|f\|_\BMO},
$$
follows immediately from the definition of $\|f\|_\BMO$.

 \section{Generalised Gagliardo--Nirenberg inequality II}\label{sec:GGN2}

 We now adapt the very elegant argument of Chen \& Zhu \cite{CZ}  to prove the following stronger version of the inequality in (\ref{interpolated}) in the case $s=n/2$; they proved the inequality for $f\in L^q\cap\BMO$, but the changes required to take $f\in L^{q,\infty}\cap\BMO$ are in fact straightforward. Another proof for $f\in L^q\cap\BMO$, which still relies on the John--Nirenberg inequality (but less explicitly), is given by Azzam \& Bedrossian \cite{AB}, and a sketch of an alternative proof of the result for $f\in L^{q,\infty}\cap\BMO$ can be found in the paper by Kozono et al.\ \cite{KMW} (see also the discussion in Section \ref{sec:ispaces}, below).

 \begin{theorem}\label{thm:BMOineq}
   For any $1\le q<p<\infty$, if  $f\in L^{q,\infty}(\R^n)\cap\BMO(\R^n)$ then $f\in L^p(\R^n)$ and there exists a constant $C=C(q,p,n)$ such that
   \begin{equation}\label{BMOineq}
   \|f\|_{L^p}\le C\|f\|_{L^{q,\infty}}^{q/p}\|f\|_\BMO^{1-q/p}.
   \end{equation}
 \end{theorem}

\begin{proof}
First we note that it is a consequence of the John--Nirenberg inequality from Lemma \ref{JN} that if $f\in\BMO\cap L^1$ then
\begin{equation}\label{JN-BMO}
d_f(\alpha)\le C\e^{-C\alpha/\|f\|_\BMO}\|f\|_{L^1}
\end{equation}
for all $\alpha>\|f\|_\BMO$; this follows by taking $|Q|\to\infty$ in (\ref{expJN}), since when $f\in L^1$,
 $$
 |f_Q|=\frac{1}{|Q|}\int|f|\to0\qquad\mbox{as}\quad|Q|\to\infty,
 $$
 and $\int_Q|f-f_Q|\,\d x\le 2\int_Q|f|\,\d x$.

Now take $f\in\BMO$ with $\|f\|_\BMO=1$. Split $f=f_{1-}+f_{1+}$ as in Lemma \ref{splitg}. Since $f_{1-}\in L^\infty$, $\|f_{1-}\|_\BMO\le 2\|f_{1-}\|_{L^\infty}\le 2$ (using (\ref{Linf2BMO})); thus $f_{1+}=f-f_{1-}\in\BMO$ and
$$
\|f_{1+}\|_\BMO\le\|f\|_\BMO+\|f_{1-}\|_\BMO\le 3.
$$

Using Lemma \ref{splitg} we know that
\begin{equation}\label{f1-}
\|f_{1-}\|_{L^p}^p\le C\|f_{1-}\|_{L^{q,\infty}}^q.
\end{equation}
Also, for $(q,q')$ conjugate,
$$
\|f_{1+}\|_{L^1}=\int|f_{1+}|\le\int|f_{1+}|^{1+1/q'}=\|f_{1+}\|_{L^{1+1/q'}}^{1+1/q'}\le c\|f_{1+}\|_{L^1}^{1/q'}\|f_{1+}\|_{L^{q,\infty}}
$$
(since $1<1+1/q'<q$ we can use weak-$L^p$ interpolation), which yields
$$
\|f_{1+}\|_{L^1}\le c\|f_{1+}\|_{L^{q,\infty}}^q.
$$
Now we calculate
\begin{align*}
\|f_{1+}\|_{L^p}^p&=p\int_0^\infty\alpha^{p-1}d_{f_{1+}}(\alpha)\,\d\alpha\\
&=p\int_0^1\alpha^{p-1}d_f(1)\,\d\alpha+p\int_1^\infty\alpha^{p-1}d_{f_{1+}}(\alpha)\,\d\alpha\\
&\le d_f(1)+p\left(\int_1^\infty\alpha^{p-1}C\e^{-C\alpha/3}\,\d\alpha\right)\|f_{1+}\|_{L^1},
\end{align*}
where we have used (\ref{dg+}), (\ref{JN-BMO}), and the fact that $\|f_{1+}\|_\BMO\le3$. Thus
\begin{equation}\label{f1+}
\|f_{1+}\|_{L^p}^p\le \|f\|_{L^{q,\infty}}^q+C\|f_{1+}\|_{L^{q,\infty}}^q\le C\|f\|_{L^{q,\infty}}^q.
\end{equation}

Adding (\ref{f1-})$^{1/p}$ and (\ref{f1+})$^{1/p}$ we obtain
$$
\|f\|_{L^p}\le C\|f\|_{L^{q,\infty}}^{q/p};
$$
(\ref{BMOineq}) follows.\end{proof}

\section{The interpolation space approach}\label{sec:ispaces}

So far we have avoided defining the two-parameter Lorentz spaces $L^{p,r}$, which involve decreasing rearrangements. In this final section we will obtain an inequality involving such spaces
\be{itsLorentz}
\|u\|_{L^{p,1}}\le C_{n,p,q}\|u\|_{L^{q,\infty}}^{q/p}\|u\|_\BMO^{1-q/p},
\ee
from which (at least for $q>1$) our two previous inequalities follow (we require $1<q<p<\infty$ in (\ref{itsLorentz}), see Theorem \ref{GNLorentz}). We will do this via the theory of interpolation spaces. Here we will not provide detailed proofs of any of the results, for the most part merely providing statements of the relevant general theory.

\subsection{Lorentz spaces}

Given a measurable function $f\:\R^n\rightarrow\R$, we have already defined and made much use of its distribution function $d_f$. We now define its decreasing rearrangement $f^*:[0,\infty)\to[0,\infty]$ as
$$
f^*(t)=\inf\{\alpha:\ d_f(\alpha)\le t\},
$$
with the convention that $\inf\varnothing=\infty$. The point of this definition is that $f$ and $f^*$ have the same distribution function,
$$
d_{f^*}(\alpha)=d_f(\alpha),
$$
but $f^*$ is a positive non-increasing scalar function. Since their distribution functions agree, we can use the identity in (\ref{df2Lp}) to show that the $L^p$ norm of $f$ is equal to the $L^p$ norm of $f^*$:
\begin{align*}
\int_{\R^n}|f(x)|^p\,\d x&=p\int_0^\infty \alpha^{p-1}d_f(\alpha)\,\d\alpha\\
&=p\int_0^\infty\alpha^{p-1}d_{f^*}(\alpha)\,\d\alpha=\int_0^\infty f^*(\alpha)^p\,\d\alpha.
\end{align*}

Given $1\le p,q\le\infty$, the Lorentz space $L^{p,q}(\R^n)$ consists of all measurable functions $f$ for which the quantity
$$
\|f\|_{L^{p,q}}:=\left(\int_0^\infty[t^{1/p}f^*(t)]^q\frac{\d t}{t}\right)^{1/q}
$$
(for $q<\infty$) or
$$
\|f\|_{L^{p,\infty}}:=\sup_{0<t<\infty}t^{1/p}f^*(t)
$$
(for $q=\infty$) is finite. It is simple to show (see Proposition 1.4.5 in Grafakos \cite{Grafakos}) that this definition agrees with our previous definition of $L^{p,\infty}$, that $L^{\infty,\infty}=L^\infty$, and that $L^{p,p}=L^p$ (the last of these, at least, is immediate).

If $r<s$ then $L^{p,r}\subset L^{p,s}$; so the largest space in this family for fixed $p$ is the weak space $L^{p,\infty}$, and the smallest is $L^{p,1}$. To see that $L^{p,r}\subset L^{p,\infty}$ for every $r$, simply observe that
\begin{align*}
t^{1/p}f^*(t)&=\left\{\frac{r}{p}\int_0^t[s^{1/p}f^*(t)]^r\frac{\d s}{s}\right\}^{1/r}\\
&\le \left\{\frac{r}{p}\int_0^t[s^{1/p}f^*(s)]^r\frac{\d s}{s}\right\}^{1/r}\\
&\le(r/p)^{1/r}\|f\|_{L^{p,r}},
\end{align*}
which yields $\|f\|_{L^{p,\infty}}\le (r/p)^{1/r}\|f\|_{L^{p,r}}$ on taking the supremum over $t>0$. Given this, if $r<q<\infty$ then, using H\"older's inequality,
$$
\|f\|_{L^{p,q}}=\left\{\int_0^t[t^{1/p}f^*(t)]^{q-r+r}\frac{\d t}{t}\right\}^{1/r}\le\|f\|_{L^{p,\infty}}^{(q-r)/q}\|f\|_{L^{p,r}}^{r/q}\le C_{p,q,r}\|f\|_{L^{p,r}}.
$$

\subsection{Interpolation spaces}

We now very briefly outline the theory of interpolation spaces; the general theory is modelled on the definition of the Lorentz spaces given above. For sustained expositions of the theory see Bennett \& Sharpley \cite{BS}, Bergh \& L\"ofstr\"om \cite{BL}, or Lundari \cite{Lunardi}.

Given two Banach spaces $X_0$ and $X_1$ that embed continuously into some parent Hausdorff topological vector space, which we term ``a compatible pair", we define the $K$-functional for each $x\in X_0+X_1$ and $t>0$ by
$$
K(x,t)=\inf\{\|x_0\|_{X_0}+t\|x_1\|_{X_1}:\ x_0+x_1=x,\ x_0\in X_0,\ x_1\in X_1\}.
$$
Then for $0<\theta<1$ and $1\le q<\infty$ we define the interpolation space $(X_0,X_1)_{\theta,q}$ as the space of all $x\in X_0+X_1$ for which
$$
\|x\|_{\theta,q}:=\left(\int_0^\infty[t^{-\theta}K(f,t)]^q\frac{\d t}{t}\right)^{1/q}
$$
is finite. Similarly, for $0\le\theta\le 1$ and $q=\infty$, the space $(X_0,X_1)_{\theta,\infty}$ is the space of all $x\in X_0+X_1$ such that
$$
\|x\|_{\theta,\infty}=\sup_{0<t<\infty}t^{-\theta}K(f,t)
$$
is finite. For all these spaces ($1\le q\le\infty$) we have the interpolation inequality
\begin{equation}\label{norm4tq}
\|f\|_{\theta,q}\le C_{\theta,q}\|f\|_{X_0}^{1-\theta}\|f\|_{X_1}^\theta
\end{equation}
(see Section 3.5 in Bergh \& L\"ofstr\"om \cite{BL}, for example).

Given the definitions of Lorentz spaces and of the interpolation spaces, it is not surprising that
$$
(L^1,L^\infty)_{1-1/p,r}=L^{p,r}
$$
for $0<\theta<1$, $1\le r\le\infty$. That one can replace $L^\infty$ here by BMO is much less obvious, but key to the `quick' proof of (\ref{itsLorentz}) that we give in this section.

\begin{theorem}[Bennett \& Sharpley]\label{BSBMO}
For $1<p<\infty$ and $1\le r\le\infty$,
$$
L^{p,r}=(L^1,\BMO)_{1-1/p,r}.
$$
\end{theorem}

\begin{proof}
See Chapter 5, Theorem 8.11, in Bennett \& Sharpley \cite{BS}. One can also find a proof of this result in the paper by Hanks \cite{Hanks}, and of a similar but slightly weaker result (with $L^p$ on the left-hand side) using complex interpolation spaces  in the paper by Janson \& Jones \cite{JJ}.
\end{proof}

We note here that the key step in the proof of this result given in Bennet \& Sharpley \cite{BS} (and in Hanks \cite{Hanks}) is a relationship between the sharp function of $f$,
$$
f^\sharp_Q(x):=\sup_{Q'\subset Q,\ Q'\ni x}\frac{1}{|Q'|}\int_{Q'}|f-f_{Q'}|,
$$
its decreasing rearrangement $f^*$, and the function $f^{**}(t):=\frac{1}{t}\int_0^tf^*(s)\,\d s$:
$$
f^{**}(t)-f^*(t)\le C(f_Q^\sharp)^*(t)\qquad 0<t<|Q|
$$
(Lemma 7.3 in Chapter 5 of Bennett \& Sharpley \cite{BS}). This also forms the main ingredient in the proof of (\ref{BMOineq}) in Kozono \& Wadade \cite{KW} (and the proof of (\ref{Lorentzone}) in Kozono et al.\ \cite{KMW}).

The inequality (\ref{BMOineq}) in fact follows simply from Theorem \ref{BSBMO} using the following `Reiteration Theorem', which allows one to identify interpolants between two interpolation spaces in terms of the original `endpoints'.

%
%

\begin{theorem}[Reiteration Theorem]\label{reiteration}
Let $(X_0,X_1)$ be a compatible pair of Banach spaces, and let $0\le\theta_0<\theta_1\le 1$ and $1\le q_0,q_1\le\infty$. Set
$$
Y_0=(X_0,X_1)_{\theta_0,q_0}\qquad\mbox{and}\qquad Y_1=(X_0,X_1)_{\theta_1,q_1}.
$$
If $0<\theta<1$ and $1\le q\le\infty$ then
$$
(Y_0,Y_1)_{\theta,q}=(X_0,X_1)_{(1-\theta)\theta_0+\theta\theta_1,q}.
$$
\end{theorem}

\begin{proof}
See Theorem 2.4 of Chapter 5 in Bennett \& Sharpley \cite{BS}, or Theorem 3.5.3 in Bergh \& L\"ofstr\"om \cite{BL}.
\end{proof}

\begin{corollary}[Generalised Gagliardo--Nirenberg with Lorentz spaces]\label{GNLorentz}

$ $\\ If $u\in L^{q,\infty}\cap\BMO$ for some $q>1$ and $q<p<\infty$, then $u\in L^{p,1}$ and there exists a constant $C_{n,p,q}$ such that
\begin{equation}\label{Lorentzone}
\|u\|_{L^{p,1}}\le C_{n,p,q}\|u\|_{L^{q,\infty}}^{q/p}\|u\|_\BMO^{1-q/p}.
\end{equation}
\end{corollary}

Note that given the ordering of Lorentz spaces, $L^{p,1}\subset L^{p,p}=L^p$ and so this result implies Theorem \ref{thm:BMOineq} in the case $q>1$.

\begin{proof}
Using Theorem \ref{BSBMO}, since $q>1$ we have
$$
L^{q,s}=(L^1,\BMO)_{1-1/q,s};
$$
set $\mathfrak B=(L^1,\BMO)_{1,\infty}$. Note that from (\ref{norm4tq}) $\|f\|_{\mathfrak B}\le C\|f\|_\BMO$. Now simply use the Reiteration Theorem to obtain
$$
L^{p,r}=(L^{q,s},{\mathfrak B})_{(1-q/p),r},
$$
from which the inequality (\ref{Lorentzone}) follows immediately using (\ref{norm4tq}).
\end{proof}

\noindent(One can use interpolation spaces to provide a proof of Theorem \ref{thm:BMOineq} that does not involve Lorentz spaces by using interpolation only with $q=\infty$ and then interpolation between weak $L^p$ spaces, see McCormick et al.\ \cite{MRR}.)

\section{Afterword: The Marcinkiewicz interpolation theorem}\label{sec:Marcin}

Although we have not needed it here, one of the main uses of weak spaces arises due to the powerful Marcinkiewicz interpolation theorem, in which bounds in weak spaces at the endpoints lead to bounds in strong spaces in between. We include here a statement of the theorem\footnote{Be aware that to fit in with our other statements throughout the paper we have swapped the traditional roles of $p$ and $q$ in Theorem \ref{thm:Marcinkiewicz}.} and some straightforward consequences.

We say $T$ is sublinear if
$$
|T(f+g)|\le|Tf|+|Tg|\qquad\mbox{and}\qquad |T(\lambda f)|\le|\lambda||Tf|
$$
almost everywhere.

\begin{theorem}\label{thm:Marcinkiewicz}
Suppose that $q_0<q_1$ and that $T$ is a sublinear map defined on $L^{q_0}+L^{q_1}$ such that for some $p_0,p_1$
$$
\|Tf\|_{L^{p_0,\infty}}\le A_0\|f\|_{L^{q_0}}\qquad\mbox{and}\qquad \|Tf\|_{L^{p_1,\infty}}\le A_1\|f\|_{L^{q_1}}.
$$
If
\begin{equation}\label{pandq}
\frac{1}{q}=\frac{1-t}{q_0}+\frac{t}{q_1}\qquad\mbox{and}\qquad\frac{1}{p}=\frac{1-t}{p_0}+\frac{t}{p_1}
\end{equation}
and $p\ge q$ then $T\:L^q\to L^p$ and there exists a constant $A_t$ such that
\begin{equation}\label{Mmagic}
\|Tf\|_{L^p}\le A_t\|f\|_{L^q}.
\end{equation}
\end{theorem}

With the restriction that $p_0\ge q_0$ and $p_1\ge q_1$ one can find an elementary proof of this theorem in Folland \cite{F}. To remove this restriction requires a more refined argument using the decreasing rearrangements introduced in Section \ref{sec:ispaces}, see Theorem 1.4.19 in Grafakos \cite{Grafakos} or Hunt \cite{Hunt}.

We now give some interesting consequences of this theorem.


\subsection{The Fourier transform on $L^p$, $1\le p\le 2$}

We saw in Section \ref{sec:FT} that $\mathscr F$ maps $L^1$ into $L^\infty$ and $L^2$ into $L^2$, so the following result is immediate.

\begin{corollary}\label{cor:FT}
For $1\le p\le 2$ the Fourier transform is a bounded linear map from $L^p$ into $L^q$, where $(p,q)$ are conjugate.
\end{corollary}

\subsection{A sharpened version of Young's inequality}

Another application is the improved version of Young's inequality that was promised in Section \ref{sec:Young}.

\begin{theorem}\label{bestYoung}
  Suppose that $1<p,q,r<\infty$. If $f\in L^{q,\infty}$ and $g\in L^r$ with
$$
\frac{1}{p}+1=\frac{1}{q}+\frac{1}{r}
$$
then $f\star g\in L^p$ with
\begin{equation}\label{w2sYoung}
\|f\star g\|_{L^p}\le c_{p,q,r}\|f\|_{L^{q,\infty}}\|g\|_{L^r}.
\end{equation}
\end{theorem}

\begin{proof}
Note that it follows from the conditions on $p,q,r$ that $p>q$.
Fix $f\in L^{q,\infty}$ with $\|f\|_{L^{q,\infty}}=1$, and consider the linear operator $T(g)=f\star g$. Since $1<p,q<\infty$ we can find $p_0<p<p_1$, $q_0<q<q_1$, and $0<t<1$ such $p_0\ge q_0$, $p_1\ge q_1$, and (\ref{pandq}) holds. Now using the weak form of Young's inequality from Proposition \ref{prop:wYoung},
$$
\|f\star g\|_{L^{p_0,\infty}}\le C\|g\|_{L^{q_0}}\qquad\mbox{and}\qquad\|f\star g\|_{L^{p_1,\infty}}\le C\|g\|_{L^{q_1}}.
$$
We can now use the Marcinkiewicz interpolation theorem to guarantee that
$$
\|f\star g\|_{L^p}\le C\|g\|_{L^q}.
$$
Since $f\star g$ is also linear in $f$, we obtain (\ref{w2sYoung}).
\end{proof}

\subsection{Endpoint Sobolev embedding, revisited}

Using Theorem \ref{bestYoung} and the fact that if $P_\alpha(x)=|x|^{-\alpha}$ then $[\hat P_\alpha](\xi)=c_{n,\alpha}P_{n-\alpha}(\xi)$ (this can be checked by simple calculation) we can give a very quick alternative proof of the endpoint Sobolev embedding, after Theorem 1.38 in Bahouri et al.\ \cite{C**}.

\begin{theorem}
For $2<p<\infty$ there exists a constant $c=c_{n,p}$ such that if $f\in\dot H^s(\R^n)$ with $s=n(1/2-1/p)$ then $f\in L^p(\R^n)$ and $\|f\|_{L^p}\le c\|f\|_{\dot H^s}$.
\end{theorem}

\begin{proof} We make the pointwise definition $\gamma(\xi)=|\xi|^s\hat f(\xi)$; since $f\in\dot H^s(\R^n)$, $\gamma\in L^2(\R^n)$. If we set $g={\mathscr F}^{-1}\gamma$ then $g\in L^2(\R^n)$ and $\|g\|_{L^2}=\|\gamma\|_{L^2}=\|f\|_{\dot H^s}$. Now,
$$
\hat f(\xi)=\frac{|\xi|^s\hat f(\xi)}{|\xi|^s}=\hat g(\xi)|\xi|^{-s},
$$
and so $f=g\star c_{n,n-s}^{-1}P_{n-s}$. Since $P_{n-s}\in L^{n/(n-s),\infty}$ and $g\in L^2$ it follows from Theorem \ref{bestYoung} that $f\in L^p(\R^n)$.\end{proof}

\subsection*{Acknowledgment}
We would like to thanks Jacob Azzam for his helpful comments, in particular for pointing out a number of references including \cite{CZ}. JCR would like to thank Franco Tomarelli for the invitation to speak in the Seminario Matematico e Fisico di Milano and subsequently to write this paper for the {\it Milan Journal of Mathematics}.

\end{document}